\documentclass[10pt,draft,reqno]{amsart}
\usepackage{amsmath,amssymb,amsthm}
\usepackage{amssymb,latexsym, bbm,comment}

\makeatletter
     \def\section{\@startsection{section}{1}%
     \z@{.7\linespacing\@plus\linespacing}{.5\linespacing}%
     {\bfseries
     \centering
     }}
     \def\@secnumfont{\bfseries}
     \makeatother
\newtheorem{theorem}{Theorem}[section]

\newtheorem{proposition}[theorem]{Proposition}

\theoremstyle{definition}

\theoremstyle{remark}

\numberwithin{equation}{section}
\newcommand{\LP}{L\'{e}vy process}
\newcommand{\R}{\mathbb{R}}

\newcommand{\nN}{n \in \mathbb{N}}

\newcommand{\C}{\mathbb{C}}

\newcommand{\cadlag}{c\`adl\`ag}

\newcommand{\g}{\mathfrak{g}}

\newcommand{\bean}{\begin{eqnarray*}}
\newcommand{\eean}{\end{eqnarray*}}

\newcommand{\la}{\langle}
\newcommand{\ra}{\rangle}

\newcommand{\G}{\widehat{G}}

\DeclareMathOperator*{\esssup}{ess\,sup}

\newcommand{\fp}{\mathfrak{p}}
\newcommand{\fa}{\mathfrak{a}}
\newcommand{\fn}{\mathfrak{n}}

\begin{document}

\date{}

\title[Convolution Semigroups Revisited]{Convolution Semigroups of Probability Measures on Gelfand Pairs, Revisited}

\author{David Applebaum}

\address{School of Mathematics and Statistics, University of
Sheffield, Hicks Building, Hounsfield Road, Sheffield,
England, S3 7RH}

\email{D.Applebaum@sheffield.ac.uk}

\subjclass[2000] {Primary 60B15; Secondary 60G51, 43A30, 43A90, 22E46}

\keywords{Gelfand pair, convolution semigroup, spherical function, spherical transform, Plancherel measure, generalised positive definite function, generalised negative definite function, Berg P--D function, Berg N--D function, L\'{e}vy--Khintchine formula, Lie group, Lie algebra, semisimple, symmetric space, stochastic differential equation}

\begin{abstract}
Our goal is to find classes of convolution semigroups on Lie groups $G$ that give rise to interesting processes in symmetric spaces $G/K$. The $K$--bi--invariant convolution semigroups are a well--studied example. An appealing direction for the next step is to generalise to right $K$--invariant convolution semigroups, but recent work of Liao has shown that these are in one--to--one correspondence with $K$--bi--invariant convolution semigroups. We investigate a weaker notion of right $K$--invariance, but show that this is, in fact, the same as the usual notion. Another possible approach is to use generalised notions of negative definite functions, but this also leads to nothing new. We finally find an interesting class of convolution semigroups that are obtained by making use of the Cartan decomposition of a semisimple Lie group, and the solution of certain stochastic differential equations. Examples suggest that these are well--suited for generating random motion along geodesics in symmetric spaces.
\end{abstract}

\maketitle

\section{Introduction}

The category of Gelfand pairs is a beautiful context in which to explore probabilistic ideas. It provides an elegant mathematical formalism, and contains many important examples, not least the globally Riemannian symmetric spaces and the homogeneous trees. Until quite recently, most studies of probability measures on a Gelfand pair $(G, K)$ have focussed on the $K$ bi--invariant case. In the context that will concern us here, where the object of study is a convolution semigroup of such measures, Herbert Heyer's paper \cite{Hey1} presents a masterly survey of the main developments of the theory, up to and including the early 1980s.

In fact, right $K$--invariant measures on $G$ are natural objects of study as they are in one--to--one correspondence with measures on the homogeneous space $N:=G/K$. The additional assumption of left $K$--invariance certainly makes the theory extremely elegant, as it enables the use of the beautiful concept of spherical function, as introduced by Harish--Chandra; thus we may study measures in the ``Fourier picture'', using the ``characteristic function'' given by the spherical transform. Such an approach led to a specific L\'{e}vy-Khintchine formula, classifying  infinitely divisible probability measures on non--compact symmetric spaces, in the pioneering work of Gangolli \cite{Gang1} (see also \cite{LW} for a more recent treatment). Another key consequence of the $K$--bi--invariance assumption is that it corresponds precisely to semigroups/Dirichlet forms that are $G$--invariant on $M$ (see Theorem 4.1 in \cite{Berg}), with respect to the natural group action; so that if $N$ is a symmetric space, then the induced semigroup commutes with all isometries.

In \cite{AD}, Dooley and the author developed a L\'{e}vy-Khintchine formula on non-compact semi-simple Lie groups, using a matrix-valued generalisation of

 \noindent Harish--Chandra's spherical functions. In the last part of the paper, an attempt was made to project this to non--compact symmetric spaces when the convolution semigroup comprises measures that are only right $K$--invariant. Since then Liao has shown \cite{L1} that all right $K$--invariant convolution semigroups are in fact $K$--bi--invariant. In the current paper, we will ask the question -- are there any natural classes of convolution semigroups, other than the $K$--bi--invariant ones, that give rise to interesting classes of Markov processes on $N$?

 Since right $K$--invariance is so natural and attractive, we begin by asking

 \noindent whether some weaker notion of convolution may lead to any interesting conclusions. We present a candidate, but find that it once again leads to $K$--bi--invariance. Our second approach is to generalise the ideas of positive--definite and negative--definite function on the space $\mathcal{P}$ of positive--definite spherical functions, which were first introduced, in the bi--invariant case, by Berg in \cite{Berg}. The functions that Berg considered were complex--valued, but when we drop left $K$--invariance, we find that their natural generalisations must be vector--valued, and this requires us to make use of certain direct integrals of Hilbert spaces over the space $\mathcal{P}$. Once again, however, we show that these objects lead to nothing new. Finally, in the semisimple Lie group case, we introduce a promising class of convolution semigroups, which are obtained by solving stochastic differential equations (SDEs). These SDEs are driven by vector fields that live in that part of the Cartan decomposition of the Lie algebra of $G$ which projects non--trivially to $N$. Although we haven't developed the ideas very far herein, this concept seems more promising. In particular there are already some interesting and non--trivial examples, which involve randomising the notion of geodesic.

The organisation of the paper is as follows. Section 2 is an introduction that briefly summarises all the harmonic analysis on Gelfand pairs that we will need in the sequel. In section 3, we discuss various types of convolution semigroup, while section 4 describes the generalised notions of vector--valued positive--definite and negative--definite function. In section 5, in the Lie group/symmetric space setting, we review Hunt's theorem and the L\'{e}vy-Khintchine formula for convolution semigroups, and section 6 puts the main results of \cite{AD} within a more general framework. Finally in section 7, we present the new class of convolution semigroups mentioned above.
\vspace{5pt}

{\bf Notation. } If $G$ is a locally compact Hausdorff group then $\mathcal{B}(G)$ is its Borel $\sigma$--algebra, and $C_{u}(G)$ is the Banach space (with respect to the supremum norm) of all real--valued, bounded, uniformly continuous functions (with respect to the left uniform structure)\footnote{In this context, uniform continuity means that given any $\epsilon > 0$, there exists a neighbourhood $U$ of $e$, so that $\sup_{x \in G}|f(g^{-1}x) - f(x)| < \epsilon$ for all $g \in U$.} defined on $G$. The closed subspace of $C_{u}(G)$ comprising functions having compact support is denoted by $C_{c}(G)$. If $\mu$ is a measure on $G$, then $\tilde{\mu}$ is the reversed measure, i.e. $\tilde{\mu}(A) = \mu(A^{-1})$, for all $A \in \mathcal{B}(G)$. We recall that if $\mu_{1}$ and $\mu_{2}$ are two finite measures on $(G, \mathcal{B}(G))$, then their convolution $\mu_{1} * \mu_{2}$ is the unique finite measure on $(G, \mathcal{B}(G))$ so that
 $$ \int_{G}f(g)(\mu_{1} * \mu_{2})(dg) = \int_{G}\int_{G}f(gh)\mu_{1}(dg)\mu_{2}(dh),$$
 for all $f \in C_{c}(G)$. If $e$ is the neutral element in $G$, then $\delta_{e}$ will denote the Dirac measure at $e$. The set $\widehat{G}$ comprises all equivalence classes (with respect to unitary conjugation) of irreducible unitary representations of $G$, acting in some complex separable Hilbert space. If $E$ is a real or complex Banach space, then $B(E)$ will denote the algebra of all bounded linear operators on $E$. If $\mathcal{F}(G)$ is some space of functions on $G$, and $K$ is a closed subgroup of $G$, we write $\mathcal{F}_{K}(G)$ for the subspace that comprising those functions that are right $K$--invariant, and we will naturally identify this subspace with the corresponding space $\mathcal{F}(G/K)$ of functions on the homogeneous space $G/K$ of left cosets.
 We choose once and for all a left--invariant Haar measure on $G$, which is denoted by $dg$ within integrals. Haar measure on compact subgroups is always normalised to have total mass one.

\section{Gelfand Pairs and Spherical Functions}

Let $(G,K)$ be a {\it Gelfand pair}, so that $G$ is a locally compact group with neutral element $e$, $K$ is a compact subgroup, and the Banach algebra (with respect to convolution) $L^{1}(K\backslash G/K)$ of $K$--bi--invariant functions is commutative. We will summarise basic facts that we will need about these structures in this section. Most of this is based on Wolf \cite{Wol}, but see also Dieudonn\'{e} \cite{Dieu}. Throughout this paper, we will, where convenient, identify functions/measures/distributions on the homogeneous space $G/K$ with right $K$--invariant functions/ measures/ distributions on $G$. We emphasise that we do not assume that left $K$--invariance also holds.

If $(G, K)$ is a Gelfand pair, then Haar measure on $G$ is unimodular. Every continuous multiplicative function from $L^{1}(K\backslash G/K)$ to $\C$ is of the form $f \rightarrow \widehat{f}(\omega)$, where $\widehat{f}(\omega) = \int_{G}f(g)\omega(g^{-1})dg$. The mapping $\omega:G \rightarrow \C$ is called a (bounded) {\it spherical function}. In general a spherical function on $(G,K)$ is characterised by the property that it is a non-trivial continuous function such that for all $g,h \in G$,

\begin{equation} \label{spher}
\int_{K}\omega(gkh)dk = \omega(g)\omega(h),
\end{equation}

The set $S(G,K)$ of all bounded spherical functions on $(G,K)$ is the maximal ideal space (or spectrum) of the algebra $L^{1}(K\backslash G/K)$. It is locally compact under the weak-$*$-topology (and compact if $L^{1}(K\backslash G/K)$ is unital). The corresponding Gelfand transform is the mapping $f \rightarrow \widehat{f}$ (usually called the {\it spherical transform} in this context). Let $\mathcal{P}:= \mathcal{P}(S,K)$ be the closed subspace of $S(G,K)$ comprising {\it positive definite} spherical functions. For each $\omega \in \mathcal{P}$, there exists a triple $(H_{\omega}, \pi_{\omega}, u_{\omega})$, where $H_{\omega}$ is a complex Hilbert space, $\pi_{\omega}$ is a unitary representation of $G$ in $H_{\omega}$, and $u_{\omega} \in H_{\omega}$ is a cyclic vector, so that for all $g \in G$,
$$ \omega(g) = \la u_{\omega}, \pi_{\omega}(g)u_{\omega} \ra.$$ Since $\omega(e) = 1$, $u_{\omega}$ is a unit vector for all $\omega \in \mathcal{P}$;
moreover the representation $\pi_{\omega}$ is irreducible, and {\it spherical} in that $\pi_{\omega}(k)u_{\omega} = u_{\omega}$, for all $k \in K$. Finally, we have dim$(H_{\omega}^{K}) = 1$, where $H_{\omega}^{K}: = \{v \in H_{\omega}, \pi_{\omega}(k)v = v~\mbox{for all}~k \in K\}.$

There is a unique Radon measure $\rho$ on $\mathcal{P}$, called the {\it Plancherel measure}, such that for all bounded functions $f \in L^{1}(K\backslash G/K), g \in G$,
$$ f(g) = \int_{\mathcal{P}}\widehat{f}(\omega)\omega(g)\rho(d\omega).$$

The support of $\rho$ is the maximal ideal space $\mathcal{P}_{+} \subset \mathcal{P}$ of the $C^{*}$-algebra $C^{*}(G, K)$, which is the uniform closure of the range of the representation $\psi$ in $B(L^{2}(K\backslash G/K))$ whose action is given by $\psi(f)h = f * h$, for all $f \in L^{1}(K\backslash G/K),$

\noindent $h \in L^{2}(K\backslash G/K)$.

The Plancherel theorem states that if $f \in L^{1}(K\backslash G/K) \cap L^{2}(K\backslash G/K)$, then $\widehat{f} \in L^{2}(P, \rho)$ and $||f||_{2} = ||\widehat{f}||_{2}$. This isometry extends uniquely to a unitary isomorphism between $L^{2}(K\backslash G/K)$ and $L^{2}(P, \rho)$.

The theory described above is essential for the analysis of $K$--bi--invariant functions/measures/distributions on $G$. To work with objects that are only right $K$--invariant we must introduce the direct integrals

\noindent $\mathcal{H}_{p}(G,K)$, for $1 \leq p \leq \infty$. These spaces comprise sections $\Psi: \mathcal{P} \rightarrow \bigcup_{\omega \in \mathcal{P}} H_{\omega}$ for which $\Psi(\omega) \in H_{\omega}$ for all $\omega \in \mathcal{P}$, such that for $1 \leq p < \infty$,
$$  ||\Psi||_{\mathcal{H}_{p}}: = \left(\int_{P}||\Psi(\omega)||_{H_{\omega}}^{p}\rho(d\omega)\right)^{\frac{1}{p}} < \infty,$$
and for $p = \infty, ||\Psi||_{\mathcal{H}_{\infty}}:= \esssup_{\omega \in \mathcal{P}}||\Psi(\omega)||_{H_{\omega}}.$

\noindent For $1 \leq p \leq \infty$, let $\mathcal{H}_{p}^{(0)}(G,K)$ be the subspace of $\mathcal{H}_{p}(G,K)$ comprising sections $\Psi$ for which $\Psi(\omega) \in H_{\omega}^{K}$ for all $\omega \in \mathcal{P}$. We will find it convenient in the sequel to regard $L^{p}(\mathcal{P}, \rho)$ as a subspace of $\mathcal{H}_{p}(G,K)$, by observing that it is precisely $\mathcal{H}_{p}^{(0)}(G,K)$.

\noindent $\mathcal{H}_{p}(G,K)$ is a Banach space, while $\mathcal{H}_{2}(G,K)$ is a Hilbert space with inner product
$$ \la \Psi_{1}, \Psi_{2} \ra_{\mathcal{H}_{2}} = \int_\mathcal{P} \la \Psi_{1}(\omega), \Psi_{2}(\omega) \ra_{H_{\omega}}\rho(d\omega),$$
for all $\Psi_{1}, \Psi_{2} \in \mathcal{H}_{2}(G,K)$.

We introduce the {\it Fourier cotransform} for any unitary representation $\pi$ of $G$, $\pi(f): = \int_{G}f(g)\pi(g)dg$, where $f \in L^{1}(G/K)$. We will need the scalar Fourier inversion formula for bounded $f \in L^{1}(G/K), g \in G$:

\begin{equation} \label{sFi}
f(g) = \int_{\mathcal{P}}\la \pi_{\omega}(f)u_{\omega}, \pi_{\omega}(g)u_{\omega} \ra \rho(d\omega),
\end{equation}
The {\it vector-valued Fourier transform} is the mapping

\noindent $\mathcal{F}:L^{1}(G/K) \rightarrow \mathcal{H}^{\infty}(G,K)$ defined for each $f \in L^{1}(G/K), \omega \in \mathcal{P}$ by
$$ (\mathcal{F} f)(\omega) = \pi_{\omega}(f) u_{\omega}.$$

The vector-valued Fourier inversion formula is a minor reformulation of (\ref{sFi}). It states that if $f \in L^{1}(G/K)$ is bounded, then $\mathcal{F}f \in \mathcal{H}_{1}(G, K)$ and for all $g \in G$:
$$ f(g) = \int_{\mathcal{P}}\la (\mathcal{F} f)(\omega), \pi_{\omega}(g)u_{\omega} \ra \rho(d\omega).$$

There is also a Plancherel formula within this context: if $f \in L^{1}(G/K) \cap L^{2}(G/K)$, then $\mathcal{F}f \in \mathcal{H}_{2}(G, K)$ and
\begin{equation} \label{Planch}
||\mathcal{F}f||_{\mathcal{H}_{2}(G, K)} = ||f||_{L^{2}(G/K)},
\end{equation}
and the action of $\mathcal{F}$ extends to a unitary isomorphism between $L^{2}(G/K)$ and $\mathcal{H}_{2}(G,K)$.

\section{Restricted Convolution Semigroups}

A family $(\mu_{t}, t \geq 0)$ of probability measures on $(G, \mathcal{B}(G))$ is said to be a {\it convolution semigroup} if $\mu_{s+t} = \mu_{s} *\mu_{t}$ for all $s,t \geq 0$. Then $\mu_{0}$ is an idempotent measure and so must be the normalised Haar measure of a compact subgroup of $G$ (see \cite{He1}, Theorem 1.2.10, p.34). A convolution semigroup is said to be {\it continuous} if $\mbox{vague}-\lim_{t \rightarrow 0}\mu_{t} = \mu_{0}$. It then follows that it is vaguely continuous on $[0, \infty)$. A continuous convolution semigroup is said to be {\it standard} if $\mu_{0} = \delta_{e}$.  If $(\mu_{t}, t \geq 0)$ is standard, then  $(P_{t}, t \geq 0)$ is a $C_{0}$-contraction semigroup on $C_{u}(G)$, where

\begin{equation}  \label{semi}
P_{t}f(g) = \int_{G}f(gh)\mu_{t}(dh), ~\mbox{for all}~ t \geq 0, f \in C_{u}(G), g \in G.
\end{equation}

It is precisely the standard continuous convolution semigroups that are the laws of L\'{e}vy processes in Lie groups (see e.g. \cite{Liao}).

Now let us return to Gelfand pairs $(G,K)$. Let $\mathcal{M}_{K}(G)$ be the space of all right $K$--invariant Radon probability measures on $G$. We say that a continuous convolution semigroup $(\mu_{t}, t \geq 0)$ is right $K$--invariant, if $\mu_{t} \in \mathcal{M}_{K}(G)$ for all $t \geq 0$. In that case, $\mu_{0}$ is normalised Haar measure on $K$, and then $\mu_{t}$ is $K$--bi--invariant for all $t \geq 0$, as is shown in Proposition 2.1 of \cite{L1}. Indeed since for all $t \geq 0, \mu_{t} = \mu_{0} * \mu_{t}$, left $K$--invariance of $\mu_{t}$ follows from that of $\mu_{0}$.

It would be desirable to be able to study families of measures on $G$ that are right $K$--invariant, but not necessarily $K$--bi--invariant, and which capture the essential features of a convolution semigroup that we need, within a right $K$--invariant framework. Here is a plausible candidate. A family $(\mu_{t}, t \geq 0)$ of probability measures on $(G, \mathcal{B}(G))$ is said to be a {\it right $K$-invariant restricted convolution semigroup} if

\begin{enumerate}
\item[(A1)] $\mu_{t}$ is right $K$--invariant for all $t > 0$,
\item[(A2)] $\mu_{0} = \delta_{e}$,
\item[(A3)] $\int_{G}f(g)\mu_{s+t}(dg) = \int_{G}\int_{G}f(gh)\mu_{s}(dg)\mu_{t}(dh)$, for all $f \in C_{u}(G/K)$,
\item[(A4)] $\lim_{t \rightarrow 0}\int_{G}f(\sigma)\mu_{t}(d\sigma) = f(e)$, for all $f \in C_{c}(G/K)$.
\end{enumerate}

\pagebreak

{\bf Notes.} \begin{enumerate}
              \item In (A4) we can replace $f$ with any bounded continuous

              right $K$--invariant function, by the argument of Theorem 1.1.9 in \cite{He1}.
              \item The term ``restricted convolution semigroup'' is a misnomer, as (A3) appears to be too weak to define a convolution of measures in the usual sense (but see Theorem \ref{bad} below). Nonetheless there is a more general framework that these ideas fit into. We regard $C_{u}(G/K)$ as a $*$--bialgebra where the comultiplation

                  $\Delta: C_{u}(G/K) \rightarrow C_{u}(G/K) \otimes C_{u}(G/K)$ is given by
              $$ \Delta f(g, h) = \int_{K}f(gkh)dk;$$
              then we can interpret (A3) as a convolution of states on $C_{u}(G/K)$, as described in e.g. \cite{Fra}.
            \end{enumerate}

Note that $P_{t}$, as defined in (\ref{semi}) does not preserve the space $C_{u}(G/K)$. Instead we define the family of operators $T_{t}:C_{u}(G/K) \rightarrow C_{u}(G/K)$, for $t \geq 0$ by

\begin{equation}  \label{semi1}
T_{t}f(g) = \int_{G}\int_{K}f(gkh)\mu_{t}(dh)dk, ~\mbox{for all}~ t \geq 0, f \in C_{u}(G), g \in G.
\end{equation}

Then by using (A1) to (A4), we can verify that $(T_{t}, t \geq 0)$ is a $C_{0}$-contraction semigroup on $C_{u}(G/K)$. Indeed, for $s, t \geq 0$, to verify the semigroup property, we observe that by right $K$--invariance of $\mu_{s}$, Fubini's theorem, and (A3):

\bean T_{s}(T_{t}f)(g) & = & \int_{G}\int_{K}(T_{t}f)(gkh)\mu_{s}(dh)dk\\
& = & \int_{G}\int_{K}\int_{G}\int_{K}f(gkhk^{\prime}h^{\prime})\mu_{t}(dh^{\prime})dk^{\prime}\mu_{s}(dh)dk\\
& = & \int_{G}\int_{K}\int_{G}f(gkhh^{\prime})\mu_{t}(dh^{\prime})\mu_{s}(dh)dk\\
& = & \int_{G}\int_{K}f(gkh)\mu_{s+t}(dh)dk\\
& = & T_{s+t}f(g). \eean

Although they appear to be promising objects, as pointed out to the author by Ming Liao, restricted convolution semigroups are just continuous $K$--bi-invariant convolution semigroups, and so there is nothing new in this idea. We prove this as follows.

\begin{theorem} \label{bad} Every right $K$--invariant restricted convolution semigroup is a continuous $K$--bi--invariant convolution semigroup.

\end{theorem}

 \begin{proof} Let $(\mu_{t}, t \geq 0)$ be a right $K$--invariant restricted convolution semigroup. For each $f \in C_{u}(G), g \in G$, define $f^{K}(g) = \int_{K}f(gk)dk$. Then $f^{K} \in C_{u}(G/K)$, and for all $s, t \geq 0$, using right--$K$--invariance of $\mu_{s+t}$, Fubini's theorem, and right $K$--invariance of $\mu_{t}$
\bean \int_{G}f(g)\mu_{s+t}(dg) & = &  \int_{G}f^{K}(g)\mu_{s+t}(dg)\\
& = & \int_{G}\int_{G}f^{K}(gh)\mu_{s}(dg)\mu_{t}(dh)\\
& = & \int_{G}\int_{G}f(gh)\mu_{s}(dg)\mu_{t}(dh).\eean

By a similar argument
\bean \lim_{t \rightarrow 0}\int_{G}f(g)\mu_{t}(dg) & = & \lim_{t \rightarrow 0} \int_{G}f^{K}(g)\mu_{t}(dg)\\
& = & f^{K}(e) = \int_{K}f(k)dk.\eean

So $(\mu_{t}, t \geq 0)$ is a continuous right $K$ invariant convolution semigroup, with $\mu_{0}$ being normalised Haar measure on $K$. Hence, by Proposition 2.1 of \cite{L1}, $\mu_{t}$ is $K$--bi--invariant for all $t > 0.$
\end{proof}

\section{Negative Definite Functions}

 We define the vector-valued Fourier transform of $\mu \in \mathcal{M}_{K}(G)$ in the obvious way, i.e.
$$ (\mathcal{F} \mu)(\omega) = \pi_{\omega}(\mu) u_{\omega},$$ where $\pi_{\omega}(\mu) = \int_{G}\pi_{\omega}(g)\mu(dg)$ for all $\omega \in \mathcal{P}$. A straightforward calculation yields:

\begin{equation}  \label{conv}
(\mathcal{F} (\mu_{1} * \mu_{2}))(\omega) = \pi_{\omega}(\mu_{1})(\mathcal{F} \mu_{2})(\omega),
\end{equation}
for all $\mu_{1}, \mu_{2} \in \mathcal{M}_{K}(G), \omega \in \mathcal{P}$. We also use the standard notation $$\widehat{\mu}(\pi_{\omega}): = \pi_{\omega}(\mu)^{*} = \int_{G}\pi_{\omega}(g^{-1})\mu(dg)$$ for the Fourier transform of an arbitrary bounded measure on $(G, \mathcal{B}(G))$.

\noindent Let $\mathcal{M}(K\backslash G/K)$ denote the space of $K$--bi--invariant Radon probability measures on $(G, \mathcal{B}(G))$. If $\omega \in \mathcal{P}$ then the {\it spherical transform} of $\mu$ is given by
$$ \widehat{\mu}_{S}(\omega) : = \int_{G}\omega(g^{-1})\mu(dg),$$
so that
$$ \widehat{\mu}_{S}(\omega) = \la \widehat{\mu}(\pi_{\omega}) u_{\omega},  u_{\omega} \ra.$$

It is shown in Theorem 6.8 of \cite{Hey1} that the mapping $\mu \rightarrow  \widehat{\mu}_{S}$ is injective.

\begin{proposition} \label{uni} The mapping $\mathcal{F}: \mathcal{M}_{K}(G) \rightarrow \mathcal{H}_{\infty}(G,K)$ is injective.
\end{proposition}

 \begin{proof} We follow the procedure of the proof of Lemma 2.1 in \cite{Berg}. First assume that $\mu$ is absolutely continuous with respect to Haar measure, and that the Radon-Nikodym derivative $\frac{d\mu}{dg} =: f \in L^{1}(G/K) \cap L^{2}(G/K)$. Then $\mathcal{F} \mu = 0$ implies that $\mathcal{F}f = 0$ and so $f = 0$ (a.e.) by (\ref{Planch}). For the general case, let $(\psi_{V}, V \in \mathcal{V})$ be an approximate identity based on a fundamental system $\mathcal{V}$ of neighbourhoods of $e$. Now define $f_{V} = \psi_{V} * \mu$. Then for all $V \in \mathcal{V}$, we see from (\ref{conv}) that $\mathcal{F} \mu = 0$ implies that $\mathcal{F}f_{V} = 0$, hence $f_{V} = 0$ (a.e.), and it follows that $\mu = 0$, as required. \end{proof}

\vspace{5pt}

In \cite{Berg}, Berg defined notions of positive- and negative-definite function that could be used to investigate $K$--bi--invariant convolution semigroups. We remind the reader of these notions. A continuous function $p: \mathcal{P} \rightarrow \C$ is said to be {\it positive definite} if $p = \widehat{\mu}_{S}(\omega)$ for some $\mu \in \mathcal{M}(K\backslash G/K)$, and a continuous function $q: \mathcal{P} \rightarrow \C$ is said to be {\it negative definite} if $q(1) = 0$ and $\exp(-tq)$ is positive definite for all $t > 0$. Berg then showed that there is a one--to--one correspondence between negative definite functions and continuous convolution semigroups in $\mathcal{M}(K\backslash G/K)$.

We extend these notions to a more general context as follows. A field $\Psi \in \mathcal{H}_{\infty}(G,K)$ is said to be {\it generalised positive definite} if there exists $\mu \in \mathcal{M}_{K}(G)$ such that $\Psi = \mathcal{F} \mu$. A densely defined closed linear operator $Q$ on $L^{2}(\mathcal{P}, \rho) = \int_{\omega \in \mathcal{P}}H_{\omega}^{K}\rho(d\omega)$
is said to be {\it generalised negative definite}  if it is diagonalisable, in that $Q = (Q(\omega), \omega \in \mathcal{P})$ with each $Q(\omega)$ acting as multiplication by a scalar in $H_{\omega}^{K}$, for $\omega \in \mathcal{P}$, and is such that
\begin{enumerate}

\item $Q(1) = 0$, where $Q(1)$ denotes the restriction of $Q$ to $H_{1} = H_{1}^{K}$.

\item $Q$ is the infinitesimal generator of a one--parameter contraction semigroup $(R_{t}, t \geq 0)$ acting on $L^{2}(\mathcal{P}, \rho)$.

\item For each $t \geq 0$, $R_{t}$ extends to a bounded linear operator on $\mathcal{H}_{\infty}(G,K)$, so that the mapping $\omega \rightarrow R_{t}(\omega)u_{\omega}$ is positive definite.


\end{enumerate}

{\bf Notes.} \begin{enumerate}
\item In (3), as $u_{\omega}$ is cyclic in $H_{\omega}$, it is equivalent to require $$R_{t}(\omega)\pi_{\omega}(g)u_{\omega} = \pi_{\omega}(\mu_{t})\pi_{\omega}(g)u_{\omega},$$
for all $g \in G,t \geq 0, \omega \in \mathcal{P}$.
\item The positive definite field $(R_{t}(\omega)u_{\omega}, \omega \in \mathcal{P})$ is uniquely determined by $Q$ as the solution of a family of initial value problems in $H_{\omega}^{K}$ for $\omega \in \mathcal{P}$:
$$ \frac{d\Psi(t)(\omega)}{dt} = Q(\omega)\Psi(t)(\omega),~\mbox{with initial condition}~\Psi(0)(\omega) = u_{\omega}.$$

\end{enumerate}

\noindent Now suppose that $\Psi$ is positive definite and that $\mu$ is $K$ bi-invariant. Then for all $\omega \in \mathcal{P}$,
\bean \la u_{\omega}, \Psi(\omega) \ra & = & \la u_{\omega}, \pi_{\omega}(\mu_{t})u_{\omega} \ra \\
& = & \int_{G}\la  u_{\omega}, \pi_{\omega}(g) u_{\omega}\ra \mu(dg)\\
& = & \int_{G} \omega(g)\mu(dg) = \widehat{\widetilde{\mu_{S}}}(\omega), \eean

\noindent and so $p(\omega): = \la u_{\omega}, \Psi(\omega)\ra$ essentially coincides with the notion of positive definite function in the bi-invariant context, as introduced by Berg in \cite{Berg}. The word ``essentially'' is included, because Berg required his positive-definite functions to be continuous. Although we could impose continuity on our mapping $\Psi$, it is then not clear to the author how to prove the next theorem.

\begin{theorem} \label{neg} There is a one-to-one correspondence between generalised negative definite functions on $\mathcal{P}$ and  right $K$--invariant restricted convolution semigroups on $G$.
\end{theorem}

 \begin{proof} Suppose that $(\mu_{t}, t \geq 0)$ is a  right $K$--invariant restricted convolution semigroup on $G$. Then $R_{t}(\omega) = \pi_{\omega}(\mu_{t})$ defines a one-parameter contraction $C_{0}$--semigroup (of positive real numbers) in $H_{\omega}^{K}$. To see that $R_{t}(\omega)$ preserves $H_{\omega}^{K}$, for all $\omega \in \mathcal{P}, t \geq 0$, note that, by Theorem \ref{bad}, $\mu_{t}$ is left $K$-invariant, and so for all $k \in K$,
 $$ \pi_{\omega}(k)R_{t}(\omega)u_{\omega} = \int_{G}\pi_{\omega}(kg)\mu_{t}(dg) = R_{t}(\omega)u_{\omega}.$$

  To verify the semigroup property, it is sufficient to observe that for all $\omega \in \mathcal{P}, f_{\omega} \in C_{u}(G/K)$, where $f_{\omega}(\cdot) = \la \pi_{\omega}(\cdot)u_{\omega}, u_{\omega} \ra$ and that for all $s, t \geq 0$, by (A3),
 \bean \la \pi_{\omega}(\mu_{s})\pi_{\omega}(\mu_{t})u_{\omega}, u_{\omega} \ra & = & \int_{G}\int_{G} \la \pi_{\omega}(gh)u_{\omega}, u_{\omega} \ra \mu_{s}(dg)\mu_{t}(dh)\\
 & = & \int_{G}\la \pi_{\omega}(g)u_{\omega}, u_{\omega} \ra \mu_{s+t}(dg)\\
 & = & \la \pi_{\omega}(\mu_{s+t})u_{\omega}, u_{\omega} \ra. \eean

\noindent For all $\psi \in \mathcal{H}_{\infty}(G,K)$, we define $R_{t}\psi(\omega) = R_{t}(\omega)\psi(\omega)$.

\noindent Then $R_{t} \in B(\mathcal{H}_{\infty}(G,K))$ for all $t \geq 0$, since
$$ ||R_{t}\psi||_{\mathcal{H}_{\infty}(G,K)} = \esssup_{\omega \in \mathcal{P}}||R_{t}(\omega)\psi(\omega)||_{H_{\omega}} \leq ||\psi||_{\mathcal{H}_{\infty}(G,K)}.$$
It is easy to see that the restriction of  $(R_{t}, t \geq 0)$ to $L^{2}(\mathcal{P}, \rho)$  is in fact a contraction $C_{0}$--semigroup; indeed to verify strong continuity, we can use Lebesgue's dominated convergence theorem to deduce that
$$ \lim_{t \rightarrow 0}\int_\mathcal{P}||\pi_{\omega}(\mu_{t})\Psi(\omega) - \Psi(\omega)||^{2}_{H_{\omega}^{K}}\rho(d\omega) = 0,$$ for all $\Psi \in L^{2}(\mathcal{P}, \rho)$.
The infinitesimal generator of $(R_{t}, t \geq 0)$ is then the required negative definite function. Note that since $\pi_{1}(\mu_{t}) = 1$ for all $t \geq 0$, it is clear that $Q(1)u_{1} = 0$.

Conversely, if $Q$ is negative definite, it follows that there exists a family of measures $\{\mu_{t}, t \geq 0\}$ in $\mathcal{M}_{K}(G)$ so that for all $\omega \in \mathcal{P}$, $Q(\omega)$ is the infinitesimal generator of the semigroup $(\pi_{\omega}(\mu_{t}), t \geq 0)$ acting in $H_{\omega}^{K}$. Hence, by (\ref{conv}) and Proposition \ref{uni}, $(\mu_{t}, t \geq 0)$ is a semigroup under convolution with $\mu_{0} = \delta_{e}$. Indeed, for all $s, t \geq 0, \omega \in \mathcal{P}$ we have
$$ \pi_{\omega}(\mu_{s+t})u_{\omega} = \pi_{\omega}(\mu_{s})\pi_{\omega}(\mu_{s})u_{\omega} = \pi_{\omega}(\mu_{s} *\mu_{t})u_{\omega}. $$

 To show that the semigroup of measures satisfies (A3), let $f \in C_{c}(G/K)$ and use scalar Fourier inversion (\ref{sFi}) as follows
\bean \int_{G}f(g)\mu_{t}(dg) & = & \int_{G}\int_{\mathcal{P}}\la \pi_{\omega}(f)u_{\omega}, \pi_{\omega}(g)u_{\omega} \ra \rho(d\omega) \mu_{t}(dg)\\
& = & \int_{\mathcal{P}}\int_{G}\la \pi_{\omega}(f)u_{\omega}, \pi_{\omega}(g)u_{\omega} \ra \mu_{t}(dg) \rho(d\omega) \\
& = & \int_{\mathcal{P}}\la \pi_{\omega}(f)u_{\omega}, \widehat{\mu_{t}}(\pi_{\omega})u_{\omega} \ra \rho(d\omega),\eean

\noindent where the use of Fubini's theorem to interchange integrals is justified by the fact that the sections $\omega \rightarrow \pi_{\omega}(f)u_{\omega} \in \mathcal{H}_{1}(G,K).$
This latter fact also justifies the use of Lebesgue's dominated convergence theorem to deduce from the above that
$$ \lim_{t \rightarrow 0}\int_{G}f(g)\mu_{t}(dg) = \int_{\mathcal{P}}\la \pi_{\omega}(f)u_{\omega}, u_{\omega} \ra \rho(d\omega) = f(e). $$ \end{proof}

\vspace{5pt}

The result of this theorem is negative. When combined with the conclusion of Theorem \ref{bad}, it tells us that there is a one--to--one correspondence between generalised negative definite functions and continuous $K$--bi--invariant convolution semigroups, and hence between generalised negative definite functions and negative definite functions. It may be interesting for future work to investigate negative definite functions that fail to be diagonalisable.

Before we leave the topic of generalised negative definite functions, we will, for completeness, follow Berg \cite{Berg}, by making intrinsic characterisations of generalised positive and negative definite functions:

We say that a field $P \in \mathcal{H}_{\infty}(G,K)$ is a {\it Berg P-D function} if for all $\nN, a_{1}, \ldots, a_{n} \in \C, \omega_{1}, \ldots, \omega_{n} \in \mathcal{P}$,

$$ \Re\left(\sum_{i=1}^{n}a_{i}\omega_{i}\right) \geq 0~\mbox{on}~G \Rightarrow \Re\left(\sum_{i=1}^{n}a_{i} \la u_{\omega_{i}}, P(\omega_{i})u_{\omega_{i}} \ra\right) \geq 0.$$

A closed densely defined linear operator $Q$ acting in $L^{2}(\mathcal{P}, \rho)$ is said to be a {\it Berg N-D function} if
\begin{enumerate}
\item $Q(1) = 0$,
\item for all $\nN, a_{1}, \ldots, a_{n} \in \C, \omega_{1}, \ldots, \omega_{n} \in \mathcal{P}$,
\bean  & & \sum_{i=1}^{n}a_{i} = 0~\mbox{and}~\Re\left(\sum_{i=1}^{n}a_{i}\omega_{i}\right) \geq 0~\mbox{on}~G\\ & \Rightarrow & \Re\left(\sum_{i=1}^{n}a_{i} \la u_{\omega_{i}}, Q(\omega_{i})u_{\omega_{i}} \ra\right) \leq 0.\eean
\end{enumerate}

We  generalise Theorem 5.1 in \cite{Berg}, where the $K$--bi--invariant case was explicitly considered:

\begin{theorem} \begin{enumerate}
\item Every generalised positive definite function on $\mathcal{P}$ is a Berg P-D function,
\item If $Q$ is a  generalised negative definite function, then $-Q$ is a Berg N-D function.
\end{enumerate}
\end{theorem}

\begin{proof} \begin{enumerate}
\item Suppose that $\Psi$ is positive definite, so that $\psi = \mathcal{F}\mu$ for some $\mu \in \mathcal{M}_{K}(G)$. For arbitrary $a_{1}, \ldots, a_{n} \in \C, \omega_{1}, \ldots, \omega_{n} \in \mathcal{P}$, we have
\bean \sum_{i=1}^{n}a_{i} \la u_{\omega_{i}}, \Psi(\omega_{i})u_{\omega_{i}} \ra & = & \int_{G} \sum_{i=1}^{n}a_{i} \la u_{\omega_{i}}, \pi_{\omega_{i}}(g)u_{\omega_{i}} \ra \mu(dg)\\
& = & \int_{G}\sum_{i=1}^{n}a_{i}\omega_{i}(g)\mu(dg), \eean
from which the required result follows easily.

\item Now suppose that $Q$ is negative definite. Then $Q$ is the infinitesimal generator of a one--parameter contraction semigroup $(R_{t}, t \geq 0)$ acting on $L^{2}(\mathcal{P}, \rho)$, and the field $ (R_{t}(\omega)u_{\omega},\omega \in \mathcal{P})$ is positive definite. Then if $\sum_{i=1}^{n}a_{i} = 0$ and $\Re\left(\sum_{i=1}^{n}a_{i}\omega_{i}\right) \geq 0$ on $G$, we see that
    $$ \Re\left(\sum_{i=1}^{n}a_{i}.\frac{1}{t}\la u_{\omega_{i}}, (R_{t}(\omega_{i}) - 1)u_{\omega_{i}} \ra\right) \geq 0$$
 and the result follows when we take the limit as $t \rightarrow 0.$
 \end{enumerate}
\end{proof}

\section{The L\'{e}vy-Khintchine Formula}

In this section $G$ is a Lie group of dimension $d$ having Lie algebra $\g$. Let $\{X_{1}, \ldots, X_{d}\}$ be a basis for $\g$, which we consider as acting as left-invariant vector fields on $G$.  We obtain a dense subspace $C_{u}^{2}(G)$ of $C_{u}(G)$ by

\bean  C_{u}^{2}(G)& := &  \left\{f \in C_{u}(G); X_{i}f \in C_{u}(G)~\mbox{and}~X_{j}X_{k}f \in C_{u}(G) \right.\\
& & \left.~\mbox{for all}~1 \leq i,j,k \leq d\right\}. \eean
It is well-known that there exist functions $x_{i} \in C_{c}(G)~(1 \leq i \leq d)$ which are canonical co-ordinate functions in a co-ordinate neighbourhood of $e$, and we say that $\nu$ is a {\it L\'{e}vy measure} on $G$ if $\nu(\{e\}) = 0$ and for any co-ordinate neighbourhood $U$ of the neutral element in $G$:

\begin{equation} \label{Lmeasdef}
\int_{G}\left(\sum_{i=1}^{d}x_{i}(\tau)^{2}\right)\nu(d\tau) < \infty~\mbox{and}~\nu(U^{c}) < \infty.
\end{equation}

\noindent where $(x_{1}, \ldots, x_{d})$ are canonical co-ordinate functions on $U$ as above.

The proof of the next celebrated theorem, goes back to the seminal work of Hunt \cite{Hunt}. The first monograph treatment was due to Heyer \cite{He1}, and more recent treatments can be found in Liao \cite{Liao}, and Applebaum \cite{App}.

\begin{theorem}[Hunt's theorem] \label{Hunt}
Let $(\mu_{t}, t \geq 0)$ be a convolution semigroup of measures in $G$, with associated semigroup of operators $(P_{t}, t \geq 0)$ acting on $C_{u}(G)$
in $G$ with generator $\mathcal{L}$ then
\begin{enumerate}
\item $C_{u}^{2}(G) \subseteq \mbox{Dom}(\mathcal{L})$. \item For each
$\sigma \in G, f \in C_{u}^{2}(G)$,
\begin{eqnarray} \label{hu}
\mathcal{L}f(\sigma) & = & \sum_{i=1}^{d}b^{i}X_{i}f(\sigma) +
\sum_{i,j=1}^{d}a^{ij}X_{i}X_{j}f(\sigma) \nonumber \\
 &  + & \int_{G-\{e\}}\left(f(\sigma \tau) - f(\sigma) -
  \sum_{i=1}^{d} x^{i}(\tau)X_{i}f(\sigma)\right)\nu(d\tau),
\end{eqnarray}

where $b = (b^{1}, \ldots b^{d}) \in {\R}^{d}, a = (a^{ij})$ is a
non-negative-definite, symmetric $d \times d$ real-valued matrix
and $\nu$ is a L\'{e}vy measure on $G$.
\end{enumerate}
Conversely, any linear operator with a representation as in
(\ref{hu}) is the restriction to $C_{u}^{2}(G)$ of the infinitesimal generator
corresponding to a unique convolution semigroup of probability
measures.
\end{theorem}

Now it is well-known that if $(\mu_{t}, t \geq 0)$ is a convolution semigroup of measures, then so is $(\widetilde{\mu_{t}}, t \geq 0)$; from which it follows that for each $\pi \in \widehat{G}, (\widehat{\widetilde{\mu_{t}}}(\pi), t \geq 0)$ is a contraction semigroup in $H_{\pi}$. Let $\mathcal{A}_{\pi}$ denote the infinitesimal generator of this semigroup, and $D_{\pi}$ be its domain. Following Heyer \cite{Hey} pp.269-70, we may extend the domain of $\mathcal{L}$ to include bounded uniformly continuous functions on $G$ that take the form $f_{\psi, \phi}(g) = \la \pi(g)\psi, \phi \ra$ for all $\psi \in D_{\pi}, \phi \in H_{\pi}, g \in G$, by observing that for all $t \geq 0$,
$$ P_{t}f_{\psi, \phi}(g) = \la \widehat{\widetilde{\mu_{t}}}(\pi)\psi, \pi(g^{-1})\phi \ra,$$

\noindent from which we can deduce that

\begin{equation} \label{inflink}
\mathcal{L}f_{\psi, \phi}(g) = \la  \mathcal{A}_{\pi}\psi, \pi(g^{-1})\phi \ra.
\end{equation}

\noindent In the sequel, we will need the infinitesimal representation $d\pi$ of $\g$, corresponding to each $\pi \in \widehat{G}$, where for each $Y \in \g, -id\pi(Y)$ is the infinitesimal generator of the strongly continuous one--parameter unitary group

\noindent $(\pi(\exp(tY)), t \in \R)$, where $\exp:\g \rightarrow G$ is the exponential map. Hence $d\pi(Y)$ is a (densely-defined) skew-adjoint linear operator acting in $H_{\pi}$.

\noindent  We will also need the dense set $H_{\pi}^{\omega}$ of analytic vectors in $H_{\pi}$ defined by
$$ H_{\pi}^{\omega}: = \{\psi \in H_{\pi}; g \rightarrow \pi(g)\psi~\mbox{is analytic}\}.$$

\noindent It is shown in \cite{Appa} that $H_{\pi}^{\omega} \subseteq D_{\pi}$.

\noindent  The following L\'{e}vy-Khintchine type formula first appeared in Heyer \cite{Hey}, where it was established for compact Lie groups. Its extension to general Lie groups is implicit in Heyer \cite{He1}. For an alternative approach, based on operator-valued stochastic differential equations, see \cite{Appa}.

\begin{theorem} \label{LK} 
If $(\mu_{t}, t \geq 0)$ is a convolution semigroup of probability measures on $G$, then for all $t \geq 0, \pi \in \G$,
$$ \widehat{\widetilde{\mu_{t}}}(\pi) = e^{t\mathcal{A}_{\pi}},$$
where for all $\psi \in H_{\pi}^{\omega}$,
\begin{eqnarray} \label{LK1}
 \mathcal{A}_{\pi}\psi & = & \sum_{i=1}^{d}b_{i}d\pi(X_{i})\psi + \sum_{j,k =1}^{d}a_{jk}d\pi(X_{j})d\pi(X_{k})\psi \nonumber \\
 & + & \int_{G}\left(\pi(\tau)\psi - \psi - \sum_{i=1}^{d}x_{i}(\tau)d\pi(X_{i})\psi\right)\nu(d\tau),
\end{eqnarray}
(where $b, a, \nu$ and $x_{i} (1 \leq i \leq d)$ are as in Theorem \ref{Hunt}.)
\end{theorem}

\begin{proof} This follows from (\ref{hu}) and (\ref{inflink}) by the same arguments as used in the proof of Theorem 5.5.1 in \cite{App}, p.145-6.

\end{proof}


\section{Convolution Semigroups on Semisimple Lie Groups and Riemannian Symmetric Pairs}

In this section we will assume that $G$ is a Lie group and that $K$ is a compact subgroup of $G$. Let ${\mathfrak k}$ denote the Lie algebra of $K$, then it is easy to see that $Xf = 0$ for all $X \in {\mathfrak k}, f \in C_{u}^{2}(G/K)$. We write the vector space direct sum $\g = {\mathfrak k} \oplus {\mathfrak k}^{\perp}$, and we choose the basis $\{X_{1}, \ldots, X_{d}\}$ of $\g$ so that $\{X_{1}, \ldots, X_{m}\}$ is a basis for ${\mathfrak k}^{\perp}$, and $\{X_{m+1}, \ldots, X_{d}\}$ is a basis for ${\mathfrak k}$.

If $G$ is semisimple, we have the {\it Iwasawa decomposition} at the Lie algebra level:
$$ \g = {\mathfrak k} \oplus \fa \oplus \fn,$$
where $\fa$ is abelian and $\fn$ is nilpotent. At the global level $G$ is diffeomorphic to $KAN$, where $A$ is abelian and $N$ is nilpotent, and we may write each $g \in G$ as
$$ g = u(g)\exp(A(g))n(g),$$
where $u(g) \in K, A(g) \in \fa$ and $n(g) \in N$ (see e.g. Chapter VI in \cite{Kn}). Any minimal parabolic subgroup of $G$ has a {\it Langlands decomposition} $MAN$ where $M$ is the centraliser of $A$ in $K$. The {\it principal series} of irreducible representations of $G$ are obtained from finite dimensional representations of $M$ by Mackey's theory of induced representations. We will say more about this below.

 A Gelfand pair $(G,K)$ is said to be a {\it Riemannian symmetric pair}, if $G$ is a connected Lie group and there exists an involutive analytic automorphism $\sigma$ of $G$ such that $(K_{\sigma})_{0} \subseteq K \subseteq K_{\sigma}$, where
 $$ K_{\sigma}: = \{k \in K; \sigma(k)= k\},$$
 and $(K_{\sigma})_{0}$ is the connected component of $e$ in $K_{\sigma}$.
 In this case, we always write ${\mathfrak p}: = {\mathfrak k}^{\perp}$, and note that
 $$ {\mathfrak p} = \{X \in \g; (d\sigma)_{e}(X) = - X\}.$$
 We also have that $N = G/K$ is a Riemannian symmetric space, under any $G$-invariant Riemannian metric on $N$, and if $\natural$ is the usual natural map from $G$ to $N$, then $(d\natural)_{e}: {\mathfrak p} \rightarrow T_{o}(X)$ is a linear isomorphism, where $o:=\natural(e)$. For details see e.g. Helgason \cite{Helg}, pp.209--10.

 If $G$ is semisimple, then we can find a Cartan involution $\theta$ of $\g$, so that $(d\sigma)_{e} = \theta$. In this case $K_{\sigma} = K$, and there is a natural Riemannian metric on $N$, that is induced by the inner product $B_{\theta}$ on $\g$, where for all $X,Y \in \g$,
 $$ B_{\theta}(X,Y) = -B(X, \theta(Y)),$$
 with $B$ being the Killing form on $\g$ (see e.g. \cite{Kn} pp.361--2).

 From now on in this section, we assume that $G$ is a noncompact, connected semisimple Lie group with finite centre, and that $K$ is a maximal compact subgroup. Then $G/K$ is a noncompact Riemannian symmetric space.
 We also assume that $G/K$ is irreducible, i.e. that the action of Ad$(K)$ on $\fp$  is irreducible. Write $g_{ij} = B(X_{i}, X_{j})$, for $i,j = 1, \ldots m$, and define the horizontal Laplacian in $G$ to be
 $$ \Delta_{H} = \sum_{i,j = 1}^{n}g_{ij}^{-1}X_{i}X_{j},$$ where $(g_{ij}^{-1})$ is the $(i,j)$th component of the inverse matrix to $(g_{ij})$.
 Then for all $f \in C_{u}^{2}(N)$, we have
 $$ \Delta_{H}(f \circ \natural) = \Delta f, $$
 where $\Delta$ is the Laplace-Beltrami operator on $N$. We also have that for each $\omega \in \mathcal{P}$, there exists $c_{\omega} > 0$ so that
 $$ \Delta_{H}\omega = -c_{\omega}\omega.$$
 It is shown in \cite{Appb} that if $(\mu_{t}, t \geq 0)$ is a $K$--bi--invariant continuous convolution semigroup, then for all $f \in C_{u}^{2}(G)$, (\ref{hu}) reduces to
 $$ \mathcal{L}f(\sigma) = a \Delta_{H}f(\sigma) + \int_{G}(f(\sigma \tau) - f(\sigma))\nu(d\tau),$$
 for all $\sigma \in G$, where $a \geq 0, \nu$ is a $K$--bi--invariant L\'{e}vy measure on $G$,  and the integral should be understood as a principal value. Then from (\ref{LK1}), we obtain Gangolli's L\'{e}vy-Khintchine formula (see also \cite{LW}), i.e. for all $t \geq 0, \omega \in \mathcal{P}$,
 $$ \widehat{\widetilde{\mu_{t}}}_{S}(\omega) = e^{-t \psi_{\omega}},$$
 where
 \bean  \psi_{\omega}: & = & \la \mathcal{A}_{\pi_{\omega}}u_{\omega}, u_{\omega} \ra \\
 & = & -a c_{\omega} + \int_{G}(\omega(g) - 1)\nu(dg). \eean

 In the remainder of this section, we will focus on more general L\'{e}vy-Khintchine formulae for standard convolution semigroups on semi-simple Lie groups.

 The spherical representations of $G$ are precisely the {\it spherical principal series}, which are obtained as follows. For each $\lambda \in \fa^{*}$, define a representation $\eta_{\lambda}$ of $M$ on $\C$ by
 $$ \eta_{\lambda}(man) = e^{-i\lambda(\xi)},$$
 where $m \in M, a = \exp(\xi) \in A, n \in N$. The required spherical representation $\pi_{\lambda}$ acting on $L^{2}(K)$ is obtained by applying the ``Mackey machine'' to $\eta_{\lambda}$. In fact, we have for each $g \in G, l \in K, f \in L^{2}(K)$,
 \begin{equation} \label{Mack}
(\xi_{\lambda}(g)f)(l) = e^{-(i\lambda - \rho)(A(lg))}f(u(lg),
\end{equation}
where $\rho$ is the celebrated half-sum of positive roots (see e.g. the Appendix to \cite{AD}), and we are using the notation $\xi_{\lambda}$ instead of $\pi_{\omega_{\lambda}}$, for a generic element of the spherical principal series.  

In this case we have $u_{\lambda}: = u_{\omega_{\lambda}} = 1$ in $L^{2}(K)$ and we obtain Harish-Chandra's beautiful formula for spherical functions:
\begin{eqnarray} \label{HCbeaut}
\omega_{\lambda}(g) & = & \la u_{\lambda}, \pi_{\lambda}(g) u_{\lambda} \ra \nonumber \\
& = & \int_{K} e^{(i\lambda + \rho)(A(kg))}dk,
\end{eqnarray}
for all $g \in G, \lambda \in \fa^{*}$. In particular, we may identify $\mathcal{P}$ with $\fa^{*}$.

For the general case, we explore the connection between the approach taken here, and the L\'{e}vy-Khintchine formula that was obtained in \cite{AD}. To that end, let $\widehat{K}$ be the unitary dual of $K$, i.e. the set of all equivalence classes (up to unitary equivalence) of irreducible representations of $K$. For each $\pi \in \widehat{K}$, let $V_{\pi}$ be the finite-dimensional inner product space on which $\pi(\cdot)$ acts, and write $d_{\pi} = \mbox{dim}(V_{\pi})$. For each $\pi_{1}, \pi_{2} \in \widehat{K}, \lambda \in \fa^{*},$ define the {\it generalised spherical function} $\Phi_{\lambda, \pi_{1}, \pi_{2}}$ by
\begin{equation} \label{supergen1}
\Phi_{\lambda,\pi_1, \pi_2}(g): = \sqrt{d_{\pi_{1}}d_{\pi_{2}}}\int_{K}e^{-(i\lambda - \rho)(A(kg))}(\pi_1(u(kg)) \otimes \overline{\pi_2}(k))dk,
\end{equation}
for all $g \in G$, where $\overline{\pi}$ denotes the conjugate representation associated to $\pi$. Hence $\Phi_{\lambda,\pi_1, \pi_2}(g)$ is a (bounded) linear operator on the space $V_{\pi_1} \otimes V_{\pi_2}^{*}$. The connection with principal series representations is made apparent in Theorem 3.1 of \cite{AD}, in that for all $g \in G, u_1, v_1 \in V_{\pi_1}, u_2, v_2 \in V_{\pi_2}$

\begin{equation} \label{keystruct}
 \la \Phi_{\lambda,\pi_1, \pi_2}(g)(u_1 \otimes u_2^*), v_1 \otimes v_2^* \ra_{V_{\pi_1} \otimes V_{\pi_2}^{*}} = \la \xi_{\lambda}(g)f_{\pi_{1}}^{u_{1},v_{1}}, f_{\pi_{2}}^{u_{2},v_{2}}) \ra_{L^{2}(K)},
\end{equation}

where for each $\pi \in \widehat{K}, u, v \in V_{\pi}, k \in K, f_{\pi}^{u,v}(k) : = \la \pi(k)u, v \ra$. Note that by Peter-Weyl theory, $\mathcal{M}(K):= \mbox{lin. span}\{f_{\pi}^{u,v}(k); \pi \in \widehat{K}, u, v \in V_{\pi}\}$ is dense in $L^{2}(K)$.

If $\mu$ is a finite measure defined on $(G, \mathcal{B}(G))$ then its {\it generalised spherical transform} is defined to be
$$ \widehat{\mu_{\lambda, \pi_{1}, \pi_{2}}}^{(S)}: = \int_{G}\Phi_{\lambda,\pi_1, \pi_2}(g^{-1})\mu(dg).$$
Then from (\ref{keystruct}), we easily deduce that

\begin{equation} \label{transrel}
\la \widehat{\mu_{\lambda, \pi_{1}, \pi_{2}}}^{(S)}(u_1 \otimes u_2^*), v_1 \otimes v_2^* \ra_{V_{\pi_1} \otimes V_{\pi_2}^{*}} = \la \widehat{\mu}(\xi_{\lambda})f_{\pi_{1}}^{u_{1},v_{1}}, f_{\pi_{2}}^{u_{2},v_{2}} \ra_{L^{2}(K)},
\end{equation}

Now replace $\mu$ by $\mu_{t}$ in (\ref{transrel}). In \cite{AD} a L\'{e}vy-Khinchine-type formula which extended Gangolli's result from \cite{Gang1} was obtained, wherein the role of the characteristic exponent was played by
$$ \eta_{\lambda, \pi_{1}, \pi_{2}}: = \frac{d}{dt}\left.\widehat{(\mu_{t})_{\lambda, \pi_{1}, \pi_{2}}}^{(S)}\right|_{t=0}.$$

Differentiating in (\ref{transrel}), we obtain
\begin{equation} \label{transre2}
\la \eta_{\lambda, \pi_{1}, \pi_{2}}(u_1 \otimes u_2^*), v_1 \otimes v_2^* \ra_{V_{\pi_1} \otimes V_{\pi_2}^{*}} = \la \mathcal{A}_{\xi_{\lambda}}f_{\pi_{1}}^{u_{1},v_{1}}, f_{\pi_{2}}^{u_{2},v_{2}} \ra_{L^{2}(K)},
\end{equation}

\noindent where we use the fact that $\mathcal{M}(K) \subseteq C^{\infty}(K) \subseteq \mbox{Dom}(\mathcal{A}_{\xi_{\lambda}})$, and from here we have a direct relationship between the L\'{e}vy-Khinchine-type formula given in Theorem 5.1 of \cite{AD}, and that of Theorem \ref{LK}.

\vspace{5pt}

In section 6 of \cite{AD} an attempt was made to use the generalised spherical transform to obtain a L\'{e}vy--Khintchine formula for right $K$--invariant convolution semigroups, in the mistaken belief that there were non--trivial elements in that class that were not $K$--bi--invariant. The work of \cite{L1}, as described in section 2 above, shows that this was erroneous.



\vspace{5pt}

\section{A New Class of Processes on Symmetric Spaces}

Let $(\Omega, \mathcal{F}, P)$ be probability space, and $(L(t), t \geq 0)$ be a right L\'{e}vy process on $G$ (so that it has stationary and independent left increments). Then the family of laws $(\mu_{t}, t \geq 0)$ is a (standard) convolution semigroup. We are interested in identifying classes of these processes so that the process $(\natural(L(t)), t \geq 0)$ on $N = G/K$, which is a Feller process (by Proposition 2.1 in \cite{Liao}, p.33), has interesting probabilistic and geometric properties. If $(\mu_{t}, t \geq 0)$ is $K$--bi--invariant, then $(\natural(L(t)), t \geq 0)$ is a L\'{e}vy process on $N$. Such processes were first investigated by Gangolli in \cite{Gang2} (see also \cite{Appb, LW}), and the generic process was shown to be a Brownian motion on $N$ interlaced with jumps having a $K$--bi--invariant distribution. We have seen that requiring that $(\mu_{t}, t \geq 0)$ is only right $K$--invariant gives us nothing new.

We begin with a \cadlag~\LP~ $(Z(t), t \geq 0)$ taking values on $\R^{m}$, where for each $t \geq 0, Z(t) = (Z_{1}(t), \ldots, Z_{m}(t))$, and having characteristics $(b, a,\nu)$. 
Assume that $G$ is semisimple and consider the global Cartan decomposition $G = \exp(\fp)K$. We induce a L\'{e}vy process $(\tilde{Z}(t), t \geq 0)$ on $\fp$ by defining $\tilde{Z}(t) = \sum_{i=1}^{m}Z_{i}(t)X_{i}$. As is shown in \cite{AK}, Corollary to Theorem 2.4, we obtain a left L\'{e}vy process $(M(t), t \geq 0)$ on $G$ by solving the stochastic differential equation (using the Markus canonical form $\diamond$):
$$ dM(t) = M(t-) \diamond d\tilde{Z}(t),$$
with initial condition $M(0) = e$ (a.s.).

The generator  takes the form

\bean
\mathcal{L}f(\sigma) & = & \sum_{i=1}^{m}b^{i}X_{i}f(\sigma) +
\sum_{i,j=1}^{m}a^{ij}X_{i}X_{j}f(\sigma) \nonumber \\
 &  + & \int_{\R^{m}}\left[f\left(\sigma \exp\left(\sum_{i=1}^{m}y^{i}X_{i}\right)\right)\right.\\ & - & \left.f(\sigma) -
  {\bf 1}_{B_{1}}(y)\sum_{i=1}^{m}y^{i}X_{i}f(\sigma)\right]\nu(dy),\eean

where $f \in C_{u}^{2}(G), \sigma \in G$. We then take $L(t) = M(t)^{-1}$ for all $t \geq 0$, to get the desired right L\'{e}vy process.


\vspace{5pt}

{\bf Example 1} Geodesics.

\vspace{5pt}

Here the process $Z$ has characteristics $(b, 0, 0)$. Fix $Y = \sum_{i=1}^{m}b_{i}X_{i} \in \fp$, and consider the deterministic L\'{e}vy process $L(t) = \exp(tY)$ for $t \geq 0$. Then the operator $\mathcal{L} = Y$ and $\natural(L(t)) = \mbox{Exp}(t d\natural(Y))o$, where Exp is the Riemannian exponential; i.e. $\natural(L(t))$ moves from $o$ along the unique geodesic having slope $d\natural(Y) \in T_{o}(N)$ at time zero.

\vspace{5pt}

{\bf Example 2} Compound Poisson Process with Geodesic Jumps.

\vspace{5pt}

Let $(W_{n}, \nN)$ be a sequence of independent, identically distributed random variables, taking values in $\fp$, and having common law $\eta$, and let $(N(t), t \geq 0)$ be a Poisson process of intensity $1$ that is independent of all the $W_{n}$'s. Consider the L\'{e}vy process defined for $t > 0$ by
$$ L(t) = \exp(W_{N(t)})\exp(W_{N(t)-1}) \cdots \exp(W_{1}),$$

The law of $L(t)$ is $\mu_{t} = e^{-t}\delta_{e} + \sum_{n=1}^{\infty}\frac{t^{n}}{n!}\eta^{*(n)}$. Then

$$ \natural(L(t)) = \mbox{Exp}(d\natural(W_{N(t)}))\circ \mbox{Exp}(d\natural(W_{N(t)-1})) \circ \cdots \circ \mbox{Exp}(d\natural(W_{1}))o,$$

describes a process which jumps along random geodesic segments. Here we slightly abuse notation so that for $X, Y \in \fp$, we write $\mbox{Exp}(d\natural(X)) \circ \mbox{Exp}(d\natural(Y))o$ for the geodesic that moves from time zero to time one, starting at the point $q = \mbox{Exp}(d\natural(Y))o$, and having slope $d\tau_{g} \circ d\natural(X)$, where $g$ is the unique element of $G$ such that $q = \tau_{g}(o):=gK$.

In this case, $(P(t), t \geq 0)$ has a bounded generator,

$$ \mathcal{L}f(\sigma) = \int_{g}(f(\sigma\exp(Y)) - f(\sigma))\tilde{\eta}(dY),$$

\noindent for $f \in C_{u}(G), \sigma \in G$. The \LP~$Z$ has characteristics $(b^{\prime}, 0, \eta^{\prime})$. Here $\eta^{\prime}:= \tilde{\eta}  \circ T$ and  $b^{\prime}_{i} = \int_{|y| < 1}y^{i}\eta^{\prime}(dy)$, where
$T$ is the vector space isomorphism between $\R^{m}$ and $\g$, which maps each element $e_{i}$ of the natural basis in $\R^{m}$ to $X_{i} (i= 1, \ldots, m)$.

\vspace{5pt}

More examples can be constructed from (1) and (2) by interlacing. These extend the results of \cite{AE} (within the symmetric space context). They can also be seen as a special case of the construction in \cite{Appold}. In is anticipated that the ideas in this section will be further developed in future work.

\vspace{5pt}

\noindent {\bf Acknowledgement}. I would like to thank Ming Liao for very helpful comments, and also the referee for some useful suggestions.

\bibliographystyle{amsplain}

\end{document}